\documentclass{amsart}
\usepackage{graphicx}
\usepackage{color}
\usepackage[colorlinks]{hyperref}
\newtheorem{thm}{Theorem}[section]
\newtheorem{lem}{Lemma}[section]

\theoremstyle{definition}

\newtheorem{example}{Example}
\theoremstyle{remark}

\numberwithin{equation}{section}

\begin{document}

\title[The converse of Baer's theorem ]{ The converse of Baer's theorem }

\author{Asadollah Faramarzi Salles}

\address{School of Mathematics and Computer Science, Damghan University, Damghan, Iran}

\email{faramarzi@du.ac.ir}%

%\thanks{}%
\subjclass{Primary 20B05; Secondary 20D15.}%
\keywords{Baer's theorem, lower central series, upper central series}%

%\date{}%
%\dedicatory{}%
%\commby{}%
% ----------------------------------------------------------------
\begin{abstract}
The Baer theorem states that for a group $G$ finiteness of $G/Z_i(G)$ implies finiteness of
$\gamma_{i+1}(G)$. In this paper we show that if $G/Z(G)$ is finitely generated then the converse is
true.
\end{abstract}

\maketitle
% ----------------------------------------------------------------
\section{Introduction}
A basic theorem of Schur (see  \cite[10.1.4]{rob}) assert that if the center of a group $G$ has finite
index, then the derived subgroup of $G$ is finite. This raises various questions: is there a generalization
to higher terms of the upper and lower central series? Is there a converse? There has been attempts to modify the statement and get conclusions. Some authors studied the situation under some extra conditions on the group. For example B. H. Neumann
\cite{neu} proved that $G/Z(G)$ is finite if $\gamma_2(G)$ is finite and $G$ is finitely generated.
This result is recently generalized by P. Niroomand \cite{nir} by proving that $G/Z(G)$ is finite if
$\gamma_2(G)$ is finite and $G/Z(G)$ is finitely generated. For generalizing to higher terms of the upper
and lower central series, R. Baer (see for example \cite[14.5.1]{rob}) has proved that, if $G/Z_i(G)$ is finite, then $\gamma_{i+1}(G)$ is finite. P. Hall (see for example \cite[14.5.3]{rob}) has proved a partial converse of Baer's theorem,that is, if $\gamma_{i+1}(G)$ is finite, then $G/Z_{2i}(G)$ is finite. In this paper we will prove that a converse of Baer's theorem when $G/Z(G)$ is finitely generated.%%*******************************************************************************************************************
\section{Results}
%%*******************************************************************************************************************
\begin{thm}\label{thm1}
Let $G$ be a finitely generated group. $\gamma
_{i+1}(G)$ is finite if and only if ${G}/{Z_{i}(G)}$ is
finite.
\end{thm}
\begin{proof}
Let $a\in G$, since $\gamma_{i+1}(G)=[\gamma_i(G), G]$ is finite,
the set of conjugates $\{a^b: b\in \gamma_i(G)\}$ is finite, so
$C_{\gamma_i(G)}(a)$ has finite index in $\gamma_i(G)$. Since $G$
is finitely generated, $\dfrac{\gamma_i(G)}{(\gamma_i(G)\bigcap
Z(G))}$ is finite. Hence $\dfrac{(\gamma_i(G)
Z(G))}{Z(G)}=\gamma_i({G}/{Z(G)})$ is finite. So by induction
$\dfrac{({G}/{Z(G)})}{Z_{i-1}({G}/{Z(G)})}$ is finite, and
then ${G}/{Z_i(G)}$ is finite. Now 14.5.1 of \cite{rob} completes the proof.
\end{proof}
%%**************************************************************************************
\begin{lem}
Let $G$ be a  group and $G/Z(G)=\langle x_1Z(G), \ldots, x_nZ(G)\rangle$.
 Then $Z(G)=\displaystyle\bigcap_{i=1}^nC_G(x_i)$.
\end{lem}\label{lem1}
\begin{proof}
It is clear that, $Z(G)$ is a subset of $C_G(x)$, for any $x\in G$. 
Now let $a$ be an element of $G$ such that, $[a, x_i]=1$ for $i=1,\ldots, n$. 
For any $b\in G$, $b=y_1y_2\cdots y_tz$ 
where, $y_i\in \{x_1, \ldots, x_n\}$ and $z\in Z(G)$. Now we have 
$[a, b]=[a, y_1y_2\cdots y_tz]=1$, since $[a, y_i]=1=[a, z]$. 
Therefore $a$ is an element of $Z(G)$.
\end{proof}
%%%***********************************************************************************************
\begin{thm}\label{thm2}
 Let $G$ be a  group and $G/Z(G)$ finitely generated. Then $\gamma
_{i+1}(G)$ is finite if and only if ${G}/{Z_{i}(G)}$ is
finite.
\end{thm}
\begin{proof} If $i=1$, then the main theorem of \cite{nir} implies the result.
Let  $i >1$ and $G/Z(G)=\langle x_1Z(G), \ldots, x_nZ(G)\rangle$. As the proof of theorem \ref{thm1},
$C_{\gamma_i(G)}(a)$ has finite index in $\gamma_i(G)$ for any $a\in G$.
Since $Z(G)=\displaystyle\bigcap_{i=1}^nC_G(x_i)$,
$\dfrac{\gamma_i(G)}{(\gamma_i(G)\bigcap
Z(G))}$ is finite. Hence $\dfrac{(\gamma_i(G)
Z(G))}{Z(G)}=\gamma_i({G}/{Z(G)})$ is finite. Now $\dfrac{G/Z(G)}{Z(G/Z(G))}=\dfrac{G}{Z_2(G)}$
is finitely generated, so by induction
$\dfrac{({G}/{Z(G)})}{Z_{i-1}({G}/{Z(G)})}$ is finite, and
then ${G}/{Z_i(G)}$ is finite. Now 14.5.1 of \cite{rob} completes the proof.
\end{proof}
%%********************************************************************************
The following example shows the finiteness conditions on the Theorem \ref{thm2} is necessary.
\begin{example}
Let $G$ be a group with generators $x_j, y_j, j>1$ and $z$, subject to the relations
$x_j^p=y_j^p=z^{p^i}=1, [x_l, x_j]=[y_l, y_j]=1$, for $k\neq j, [x_k, y_j]=1$, 
$[x_j, y_j]=z$ and
$[z, t_1,\ldots, t_r]=z^{p^{r}}$ where $t_s\in \{x_j, y_j\}$ for $s=1,\ldots, r$
and $1\leq  r\leq i-1$. Then $Z_i(G)=\langle z\rangle$ and
$\gamma_{i+1}(G)=\langle z^{p^{i-1}}\rangle$, but $G/Z_i(G)$is infinite.
\end{example}
%%**********************************************************************
%%**********************************************************************
% ----------------------------------------------------------------
%\bibliographystyle{amsplain}

%\bibliography{9}

% ----------------------------------------------------------------
%\bibliographystyle{amsplain}
%\bibliography{}
\end{document}